\documentclass[11pt]{amsart}
\usepackage{color}
\usepackage{amssymb,nicefrac,bm,upgreek,mathtools,verbatim,enumerate}
\usepackage[mathscr]{eucal}
\usepackage{graphicx}
\usepackage{epsf,times}

\definecolor{dmagenta}{rgb}{.4,.1,.5}
\definecolor{007}{rgb}{.0,.0,.7}
\definecolor{dred}{rgb}{.5,.0,.0}
\definecolor{dgreen}{rgb}{.0,.5,.0}
\definecolor{dblue}{rgb}{.0,.0,.5}

\definecolor{violet}{rgb}{.3,.0,.9}
\definecolor{orange}{cmyk}{0,.5,.1,.0}
\definecolor{dcyan}{cmyk}{.5,.0,.0,.0}
\definecolor{dyellow}{cmyk}{.0,.0,.5,.0}
\definecolor{cm}{cmyk}{1,.0,.0,.0}
\numberwithin{equation}{section}
\allowdisplaybreaks

\newtheorem{theorem}{Theorem}[section]
\newtheorem{lemma}{Lemma}[section]

\theoremstyle{definition}
\newtheorem{definition}{Definition}[section]

\theoremstyle{remark}

\newcommand{\cA}{\mathcal{A}}

\newcommand{\D}{\mathrm{d}}

\newcommand{\RR}{\mathbb{R}}
\newcommand{\bS}{\mathbb{S}}

\newcommand{\RN}{\mathbb{R}^N}

\newcommand{\cC}{\mathcal{C}}

\newcommand{\calG}{\mathcal{G}}

\newcommand{\cO}{\mathcal{O}}

\newcommand{\sL}{\mathscr{L}}

\newcommand{\bvnorm}[1]{[\kern-0.45ex[\kern0.1ex #1 \kern0.1ex]\kern-0.45ex]}

\newcommand{\abs}[1]{\lvert#1\rvert}

\newcommand{\infdel}{\Delta_\infty}

\title[Strong Maximum Principle and Compact Support Principle]
{A Strong Maximum Principle and a Compact Support Principle for infinity Laplacian}

\author[A. Biswas]{Anup Biswas}
\address{Indian Institute of Science Education and Research, Dr.\ Homi Bhabha Road, Pashan, Pune 411008}
\email{anup@iiserpune.ac.in}



\begin{document}

\begin{abstract}
In this article we find necessary and sufficient conditions for the strong maximum principle and compact support principle for non-negative solutions to the
quasilinear elliptic inequalities
$$\infdel u + G(|Du|) - f(u)\,\leq 0\quad \text{in}\; \cO,$$
and 
$$\infdel u + G(|Du|) - f(u)\,\geq 0\quad \text{in}\; \cO,$$
where $\infdel$ denotes the infinity Laplacian, $G$ is an
appropriate continuous function and $f$ is a nondecreasing, continuous
function with $f(0)=0$.
\end{abstract}

\maketitle

\section{Introduction}
The strong maximum principle of second order elliptic partial differential equations 
is due to Eberhard Hopf and it is one of the fundamental results in theory of differential equations. A very complete account of the developments in the area of maximum principles can be found in the works of Pucci and Serrin \cite{PS04,PS07},
 where a thorough discussion and a complete bibliography is presented.

In this article we are interested in maximum principles for inequalities involving
two operators $\sL_1$ and $\sL_0$ where
$$\sL_1 u=\infdel u= 
\sum_{i, j}\partial_{x_i}u\, \partial_{x_i x_j} u\, \partial_{x_j}u
\quad \text{and}\quad \sL_0 u=\frac{1}{|Du|^2}\infdel u.$$
$\sL_1$ is popularly known as the {\it infinity Laplacian} and $\sL_0$ is known
as the {\it normalized infinity Laplacian}. Though there are other variants of
infinity Laplacian operators one could consider, these two operators in particular, have received 
more attention in the literature. Infinity Laplacian was first introduced in the pioneering works of G. Aronsson \cite{AG1,AG2,AG3} and became quite popular in
the theory of partial differential equations. For more details on infinity Laplace 
operator we refer the readers to \cite{ACJ,Lindqvist}.
To introduce our problem we consider a domain $\cO$ in $\RN$. Let $v$ be a
non-negative solution to
\begin{equation}\label{E1.1}
\sL_1v - K|Dv|^3 - f(v)\leq 0 \quad \text{in}\; \cO,
\end{equation}
or 
\begin{equation}\label{E1.2}
\sL_0v - K|Dv| - f(v)\leq 0 \quad \text{in}\; \cO,
\end{equation}
where $K\geq 0$ is a constant.
By a (sub or super) solution we always mean a viscosity (sub or super) solution (see Definition~\ref{D2.1}).
In this article, $f:[0, \infty)\to [0, \infty)$ is a given continuous, non-decreasing function and $f(0)=0$.
Let $F(t)=\int_0^t f(s)\, \D{s}$. A nonnegative solution $v$ is
said to satisfy a strong maximum principle (SMP)
in $\cO$ if $v(x_0)=0$ for some $x_0\in \cO$ implies that $v\equiv 0$ in $\cO$.
We establish the following SMP for \eqref{E1.1} and \eqref{E1.2}.

\begin{theorem}\label{T1.1}
Consider the following two conditions: for some (and thus for all) $\delta>0$ we have
\begin{align}
\int_0^\delta \frac{1}{[F(s)]^{\frac{1}{4}}}\, \D{s}\,=\,\infty.\label{EL1.1A}
\\
\int_0^\delta \frac{1}{[F(s)]^{\frac{1}{2}}}\, \D{s}\,=\,\infty.\label{EL1.1B}
\end{align}
Then the following hold.
\begin{itemize}
\item[(a)] Assume \eqref{EL1.1A} holds. If $v\gneq 0$ is a solution of \eqref{E1.1}, 
then $v>0$ in $\cO$.
\item[(b)] Assume \eqref{EL1.1B} holds. If $v\gneq 0$ is a solution of 
\eqref{E1.2},
then $v>0$ in $\cO$.
\end{itemize}
\end{theorem}
To compare the above result with the existing  results let us consider the
$p$-Laplacian operator of the form
\begin{equation}\label{p-lap}
\mathrm{div}(|Dv|^{p-2} Dv)- f(v)\leq 0\quad \text{in}\; \cO.
\end{equation}
It was proved by V\'{a}zquez in \cite{V84} that $v$ in \eqref{p-lap} satisfies a 
SMP if for some $\delta>0$ we have
\begin{equation}\label{p-con}
\int_0^\delta\frac{1}{[F(s)]^{\nicefrac{1}{p}}}\, \D{s}=\infty.
\end{equation}
It turns out that this condition is also necessary for the validity of SMP; see Benilan-Brezis-Crandall \cite{BBC} for $p=2$ and Diaz \cite{D85} for all $p>1$. These results are then 
extended by Pucci, Serrin and Zou \cite{PSZ} and
by Pucci and Serrin \cite{PS04} for operators of the form 
$$\mathrm{div}(A(|Dv|) Dv)- f(v)\leq 0\quad \text{in}\; \cO,$$
for a suitable continuous function $A$. For further developments in this direction 
we refer to the works of Felmer-Montenegro-Quaas \cite{FMQ}, Felmer-Quaas \cite{FQ02}.
Our Theorem~\ref{T1.1} extends the SMP for infinity Laplacian operators.

It is shown in \cite{PSZ} that when \eqref{p-con} fails, that is, 
$$\int_0^\delta\frac{1}{[F(s)]^{\nicefrac{1}{p}}}\, \D{s}\,<\,\infty, 
\quad \text{for some}\; \delta>0,$$
then a compact support principle (CSP) holds in the sense that any 
nonnegative solution $u$ of 
$$\mathrm{div}(|Du|^{p-2} Du)- f(u)\geq 0\quad \text{in}\; B^c(0, r)$$
which also vanishes at infinity, must vanish outside a compact set ( see also \cite{FMQ,FQ02,PS04}). 
Our next result is about CSP which states that any nonnagtive
solution of 
$$\sL_i u + G(|Du|) - f(u)\geq 0 \quad \text{in}\; \cO,$$
that vanishes at infinity, must have a compact support. Here $G:[0, \infty)\to [0, \infty)$ is a continuous, nondecreasing function with $G(0)=0$.
We prove a stronger version of the CSP where we do not
assume the solution to vanish at infinity.
\begin{theorem}\label{T1.2}
Suppose that $\cO$ is unbounded and $B^c(0, \hat{r})\subset \cO$ for some $\hat{r}>0$. Let 
$f(s)>0$ for $s>0$. Then the following hold.
\begin{itemize}
\item[(a)] Define $\Gamma(t)=\int_0^{2t} G(s)\D{s} + \frac{1}{4} t^4$ and assume that
\begin{equation}\label{ET2.2A0}
\int_0^1 \frac{1}{\Gamma^{-1}(F(s))}\, \D{s} < \infty.
\end{equation}
Let $u$ be a nonnegative, bounded function that solve 
\begin{equation}\label{EqnA}
\sL_1 + G(|Du|) - f(u)\geq 0 \quad \text{in}\; \cO\,.
\end{equation}
Then there exists $R>0$ such that $u(x)=0$ for $|x|\geq R$. 

\item[(b)] Define $\Gamma(t)=\int_0^{2t} G(s)\D{s} + \frac{1}{2} t^2$ and assume that
\begin{equation}\label{ET2.2B0}
\int_0^1 \frac{1}{\Gamma^{-1}(F(s))}\, \D{s} < \infty.
\end{equation}
Let $u$ be a nonnegative, bounded function that solve
\begin{equation*}
\sL_0 + G(|Du|) - f(u)\geq 0 \quad \text{in}\; \cO.
\end{equation*}
Then there exists $R>0$ such that $u(x)=0$ for $|x|\geq R$. 
\end{itemize}
\end{theorem}

The boundedness assumption in Theorem~\ref{T1.2} can not be relaxed. For instance,
take $u(x)=e^{|x|}$, $G(s)=s^3$ and $f(s)=s^{3\alpha}$ for $\alpha\in (0, 1)$. Then an
easy calculation reveals  that for $x\neq 0$
\begin{align*}
\infdel u (x) + |Du(x)|^3 - f(u(x))= 2 e^{3|x|} - e^{3\alpha|x|}>0.
\end{align*}

Next we prove existence of a nonnegative solution with compact support. To compare
it with Theorem~\ref{T1.2} take $G(s)=Ks^3$ in Theorem~\ref{T1.3}(a) and $G(s)=K s$ in (b) below.

\begin{theorem}\label{T1.3}
Let $\cO=B^c(0, 1)$ and 
$f(s)>0$ for $s>0$. Then the following hold.
\begin{itemize}
\item[(a)] Suppose that
\begin{equation}\label{ET2.3A}
\int_0^1 \frac{1}{(F(s))^{\nicefrac{1}{4}}}\, \D{s} < \infty.
\end{equation}
Then for every $K>0$, there exists a $u\gneq 0$ with compact support satisfying
\begin{equation}\label{AB1}
\sL_1u + K |Du|^3 - f(u)= 0 \quad \text{in}\; \cO\,.
\end{equation}

\item[(b)] Suppose that
\begin{equation}\label{ET2.3B}
\int_0^1 \frac{1}{(F(s))^{\nicefrac{1}{2}}}\, \D{s} < \infty.
\end{equation}
Then for every $K>0$, there exists a $u\gneq 0$ with compact support satisfying
\begin{equation}\label{AB2}
\sL_0u + K |Du| - f(u)= 0 \quad \text{in}\; \cO\,.
\end{equation}
\end{itemize}
\end{theorem}
 We note that the solution $u$ of \eqref{AB1} (\eqref{AB2}) also satisfies
$$\sL_1u  - f(u)\leq 0 \quad \text{in}\; \cO, 
(\sL_0 u - f(u)\leq  0 \quad \text{in}\; \cO, \text{respectively}).$$
Thus Theorem~\ref{T1.3} also establishes the necessity of the conditions \eqref{EL1.1A} and \eqref{EL1.1B} in Theorem~\ref{T1.1}.

Before conclude this section let us also mention the works \cite{BD99,BB01,BCPR,SP} which
also consider maximum principles for infinity Laplacian operators. However, our maximum principles are quite different from the one studied in these works.
On the other hand, though infinite
Laplacian can not be written in a divergence form, many ideas from \cite{PSZ,FMQ}
still works for our model. The proofs of our results relies on two ingredients:
the ode method of \cite{PSZ} and a new comparison theorem for infinity Laplacian
recently obtained by Biswas and Vo in \cite{BV20,BV20a}.

\section{Proofs of Theorems~\ref{T1.1},~\ref{T1.2} and \ref{T1.3}}\label{S-proof}
We provide proofs of Theorems~\ref{T1.1},~\ref{T1.2} and \ref{T1.3} in this section.
We begin with the definition of viscosity solution. Denote by
$$\widehat\sL_1 u=\sL_1 u + H(x, Du), \quad \text{and}\quad
\widehat\sL_0 u =\sL_0 u + H(x, Du)\,,$$
where $H$ is a continuous function.
As mentioned before, in this article we deal with  viscosity solutions to the equations of the form
\begin{equation}\label{E2.1}
\widehat\sL_i u + \ell(x, u)\,=\, 0\quad \text{in}\;  \mathcal{O}, \quad \text{and}
\quad u=g\quad \text{on}\; \partial\mathcal{O},
\end{equation}
where $\ell$ and $g$ are assumed  to be continuous and $i=1,2$. For a symmetric matrix $A$ we define
$$M(A)=\max_{\abs{x}=1} \langle x, A x\rangle, \quad m(A)= \min_{\abs{x}=1} \langle x, A x\rangle.$$
The open ball of radius $r$ centered at $z$ is denoted by $B(z, r)$. We use the notation $u\prec_{z}\varphi$ when $\varphi$ touches $u$ from above exactly at the point $z$ i.e.,
for some open ball $B(z, r)$ around $z$ we have $u(y)<\varphi(y)$ for $y\in B(z, r)\setminus\{z\}$ and
$u(z)=\varphi(z)$.
\begin{definition}[Viscosity solution]\label{D2.1}
An upper-semicontinuous (lower-semicontinous) function $u$ in $\bar{\mathcal{O}}$ is said to be a viscosity sub-solution (super-solution) of \eqref{E2.1}
, written as $\sL_i u + \ell(x, u)\geq 0$ ($\sL_i u + \ell(x, u)\leq 0$),   if the followings are satisfied :
\begin{itemize}
\item[(i)] $u\leq g$ on $\partial \mathcal{O}$ ($u\geq g$ on $\partial \mathcal{O}$);
\item[(ii)] if $u\prec_{x_0}\varphi$ ($\varphi\prec_{x_0} u$ ) for
some point $x_0\in\mathcal{O}$ and a $\cC^2$ test function $\varphi$, then 
\begin{align*}
& \widehat\sL_i \varphi(x_0) + \ell(x_0, u(x_0))\,\geq\, 0\,,
\quad \left(\widehat\sL_i \varphi(x_0) + \ell(x_0, u(x_0))\leq\, 0,\; resp., \right);
\end{align*}
\item[(iii)] for $i=0$, if $u\prec_{x_0}\varphi$ ($\varphi\prec_{x_0} u$) and $D\varphi(x_0)=0$ then 
\begin{align*}
& M(D^2\varphi(x_0)) + H(x,  D \varphi(x_0)) + \ell(x_0, u(x_0))\,\geq\, 0\,,
\\
&\left(m(D^2\varphi(x_0)) + H(x,  D \varphi(x_0)) + \ell(x_0, u(x_0))\leq\, 0,\; resp., \right)\,.
\end{align*}
\end{itemize}
We call $u$ a viscosity solution if it is both sub and super solution to \eqref{E2.1}.
\end{definition}
As well known, one can replace the requirement of strict maximum (or minimum) above by non-strict maximum (or minimum). We would also require the notion of superjet and subjet from \cite{CIL}. A second order \textit{superjet} of $u$ at $x_0\in\mathcal{O}$ is defined as
$$J^{2, +}_\mathcal{O} u(x_0)=\{(D\varphi(x_0), D^2\varphi(x_0))\; :\; \varphi\; \text{is}\; \cC^2\; \text{and}\; u-\varphi\; \text{has a maximum at}\; x_0\}.$$
The closure of a superjet is given by
\begin{align*}
\bar{J}^{2, +}_\mathcal{O} u(x_0)&=\Bigl\{ (p, X)\in\RN\times\bS^{d\times d}\; :\; \exists \; (p_n, X_n)\in J^{2, +}_\mathcal{O} u(x_n)\; \text{such that}
\\
&\,\qquad  (x_n, u(x_n), p_n, X_n) \to (x_0, u(x_0), p, X)\Bigr\}.
\end{align*}
Similarly, we can also define closure of a subjet, denoted by $\bar{J}^{2, -}_\mathcal{O} u$. See \cite{CIL} for more details.

Let $H:[0, \infty)\to \RR$ be a continuous function. Denote by
$\calG_i=\sL_i + H(|Du|)$. Our next ingredient is the following comparison principle which is a special case of \cite[Theorem~2.1]{BV20}.
\begin{lemma}\label{L2.1}
Let $\mathcal{O}$ be a bounded domain and $h, \tilde{h}:\mathcal{O}\to \RR$ be 
continuous functions with $h>\tilde{h}$ in $\cO$. Suppose that 
$\calG_i v - f(v)\leq \tilde{h}$ in $\mathcal{O}$ 
and $\calG_i u - f(u)\geq h$ in $\mathcal{O}$. Then $v\geq u$ on $\partial \mathcal{O}$
implies $v\geq u$ in $\bar{\mathcal{O}}$.
\end{lemma}

\begin{proof}
As mentioned before, the proof follows from \cite[Theorem~2.1]{BV20}. We just
provide a sketch of the proof here. Suppose, on the contrary, that 
$M=\max_{\bar{\mathcal{O}}} (u-v)>0$. Now consider the coupling function
$$w_\varepsilon(x, y)= u(x)-v(y) -\frac{1}{4\varepsilon}|x-y|^4,\quad
x, y\in\bar\cO\,.$$
Let $M_\varepsilon$ be the maximum of $w_\varepsilon$ and 
$w_\varepsilon(x_\varepsilon, y_\varepsilon)=M_\varepsilon$.
It is then standard to show that (cf. \cite[Lemma~3.1]{CIL})
$$\lim_{\varepsilon\to 0} M_\varepsilon=M \quad \text{and}\quad
\lim_{\varepsilon\to 0}\frac{1}{4\varepsilon} |x_\varepsilon-y_\varepsilon|^4=0.$$
Thus, without any loss of generality, we may assume that 
$x_\varepsilon, y_\varepsilon\to z\in\bar\cO$ as $\varepsilon\to 0$. Otherwise, we may
choose a subsequence. Since $u-v\leq 0$ on $\partial\cO$ we must have $z\in\cO$.
Denote by 
$\eta_\varepsilon=\frac{1}{\varepsilon}|x_\varepsilon-y_\varepsilon|^2(x_\varepsilon-y_\varepsilon)$ and $\theta_\varepsilon(x, y)= \frac{1}{4\varepsilon}|x-y|^4$.
It then follows from \cite[Theorem~3.2]{CIL} that for some $X, Y\in\bS^{d\times d}$ we have $(\eta_\varepsilon , X)\in\bar{J}^{2, +}_\mathcal{O} u(x_\varepsilon)$,
$(\eta_\varepsilon, Y)\in\bar{J}^{2, -}_\mathcal{O} v(y_\varepsilon)$ and
\begin{equation}\label{EL1.2A}
\begin{pmatrix}
X & 0\\
0 & -Y
\end{pmatrix}
\leq 
D^2\theta_\varepsilon(x_\varepsilon, y_\varepsilon) + \varepsilon [D^2\theta_\varepsilon(x_\varepsilon, y_\varepsilon)]^2.
\end{equation}
In particular, we get $X\leq Y$. Moreover, if $\eta_\varepsilon=0$, we have $x_\varepsilon= y_\varepsilon$. Then from \eqref{EL1.2A} it follows that
\begin{equation}\label{EL1.2B}
\begin{pmatrix}
X & 0\\
0 & -Y
\end{pmatrix}
\leq 
\begin{pmatrix}
0 & 0\\
0 & 0
\end{pmatrix}.
\end{equation}
Note that \eqref{EL1.2B} implies that $X\leq 0\leq Y$ and therefore, $M(X)\leq 0\leq m(Y)$.
Applying the definition of superjet and subjet
on $\calG_1$ we now obtain for $\eta_\varepsilon\neq 0$
\begin{align}\label{EL1.2C}
h(x_\varepsilon)&\leq \langle \eta_\varepsilon X, \eta_\varepsilon \rangle
 + H(|\eta_\varepsilon|)-f(u(x_\varepsilon))\nonumber
\\
& \leq  \langle \eta_\varepsilon Y, \eta_\varepsilon \rangle
 + H(|\eta_\varepsilon|) - f(u(x_\varepsilon))\nonumber
\\
&\leq \langle \eta_\varepsilon Y, \eta_\varepsilon \rangle
 + H(|\eta_\varepsilon|) - f(v(y_\varepsilon))\nonumber
\\
&\leq \tilde{h}(y_\varepsilon),
\end{align}
where in the third line we use the fact $f(v(y_\varepsilon))\leq f(u(x_\varepsilon))$.
Letting $\varepsilon\to 0$ we obtain $h(z)\leq \tilde{h}(z)$ which is a contradiction 
to our hypothesis.
Similar argument also works  for $\calG_0$. This completes the proof.
\end{proof}

Following lemma is a key ingredient in the proof of
Theorem~\ref{T1.1}.
\begin{lemma}\label{L1.1}
For every $\varepsilon, K>0$ and $R\in (0, 1)$ we have $\alpha<0$ so that
\begin{itemize}
\item[(i)] under \eqref{EL1.1A}, there is a twice continuously differentiable solution
$\varphi$ satisfying
\begin{align}
((\varphi^\prime)^3)^\prime + K (\varphi^\prime)^3 - f(\varphi)+\alpha&=0\quad
\text{in}\; ({R}/{2}, R+\varepsilon_1),\quad
\varphi^\prime(R)=\alpha, \quad \varphi(R)=0\,,\label{EL1.1C}
\\
&0<\varphi<\varepsilon, \quad \varphi^\prime<0\quad \text{in} \; ({R}/{2}, R)\,,\label{EL1.1D}
\end{align}
for some $\varepsilon_1>0$.

\item[(i)] under \eqref{EL1.1B}, there is a twice continuously differentiable solution
$\varphi$ satisfying
\begin{align}
\varphi^{\prime\prime} + K \varphi^\prime - f(\varphi)+\alpha &=0\quad
\text{in}\; ({R}/{2}, R+\varepsilon_1),\quad
\varphi^\prime(R)=\alpha, \quad \varphi(R)=0\,,\label{EL1.1E}
\\
&0<\varphi<\varepsilon, \quad \varphi^\prime<0\quad \text{in} \; ({R}/{2}, R)\,,\label{EL1.1F}
\end{align}
for some $\varepsilon_1>0$.
\end{itemize}
\end{lemma}

\begin{proof}
We only find $\varphi$ satisfying \eqref{EL1.1C}-\eqref{EL1.1D} and the proof for
\eqref{EL1.1E}-\eqref{EL1.1F} would be analogous. The
proof of \eqref{EL1.1C}-\eqref{EL1.1D} actually follows from the argument of \cite[Lemma~2]{PSZ}. Nevertheless, we provide a proof to keep the article self-contained.
Denote by $f_\alpha = f-\alpha$  for $\alpha<0$. Also, we extend the domain $f_\alpha$ 
to $\RR$ by setting $f_\alpha(x)=-\alpha$ for $x<0$.
First we note that existence of a local solution of \eqref{EL1.1C} follows from 
the Schauder-Tychonoff fixed point theorem. In fact, for any $(\xi, \gamma)\in \RR\times\RR$,
consider the map $T:\cC[t_0-\beta, t_0] \to \cC[t_0-\beta, t_0]$ defined as
\begin{equation}\label{EL1.1G}
(Tg)(t) = \xi - \int_t^{t_0} \left(e^{K(t_0-s)}\gamma^3 - 
\int_{s}^{t_0} e^{K(\zeta-s)} f_\alpha(g(\zeta))\, \D\zeta\right)^{\nicefrac{1}{3}} \D{s}\,.
\end{equation}
Now given positive $M_1, M_2$ we can find $\beta>0$ so that
for any $|\xi|\leq M_1$, $|\gamma|\leq M_2$, $T$ satisfies
the condition of Schauder-Tychonoff fixed point theorem and hence, it has a fixed point. Set $\xi=0, \gamma=\alpha$ and find a fixed point $\varphi$ of $T$ in $[R-\beta/2, R+\beta/2]$.
Next, setting $t_0=R-\beta/2, \xi=\varphi(t_0), \gamma=\varphi'(t_0)$, we can
extend $\varphi$ to $(R-\beta, R+\beta/2)$ provided $|\varphi(t_0)|\leq M_1$
and $|\varphi'(t_0)|\leq M_2$. Let $(R_0, R+\beta/2)$ be the maximal interval on
which $\varphi$ can be defined by repeating the above scheme. It is evident that 
$\varphi$ is continuously differentiable and 
$$\varphi'(t) = \left(e^{K(R-t)}\alpha^3 - 
\int_{t}^{R} e^{K(\zeta-s)} f_\alpha(\varphi(\zeta))\, \D\zeta\right)^{\nicefrac{1}{3}}<0.
$$
Thus $\varphi$ is strictly deceasing and $\varphi^\prime< 0$ in $(R_0, R+\beta/2)$.
It is then easily seen from \eqref{EL1.1G} that $\varphi$ is twice continuously differentiable and satisfies \eqref{EL1.1C} in $(R_0, R+\beta/2)$. Thus $\varphi$
satisfies \eqref{EL1.1C} in $(R_0, R+\beta/2)$. Let $R_1\in(R, R_0]$ be the maximal
number so that $\varphi$ satisfies \eqref{EL1.1C}-\eqref{EL1.1D} in $(R_1, R+\beta/2)$.
To complete the proof we only need show that if we choose $|\alpha|$ small enough then we can have $R_1\leq R/2$. Suppose, on the contrary, that $R_1> R/2$. Given the maximality of $R_1$, one of the following to possibilities hold:
\begin{equation}\label{EL1.1H}
\mathrm{(a)}\, \lim_{t\to R_1+} \varphi(t)=\varepsilon,
\quad \mathrm{(b)}\, \lim_{t\to R_1+} |\varphi^\prime(t)|> M_2.
\end{equation}
It is easily seen from \eqref{EL1.1C} that $\varphi^{\prime\prime}>0$ in $(R_1, R)$ and
thus $\varphi^\prime$ is increasing.
Letting 
$$F_\alpha(t) =\int_0^t f_\alpha(s) \D{s},$$
and multiplying \eqref{EL1.1C} by $\varphi^\prime$ we see that
$$\left(e^{\tilde{K} t} (\varphi^\prime)^4\right)^\prime - \frac{4}{3}e^{\tilde{K} t}(F_\alpha(\varphi))^\prime=0,$$
where $\tilde{K}=\frac{4}{3}K$. Since 
$(F_\alpha(\varphi))^\prime=f_\alpha(\varphi) \varphi^\prime<0$, we get 
$$e^{\tilde{K} R} \alpha^4- e^{\tilde{K} t} (\varphi^\prime(t))^4 +
\frac{4}{3}e^{\tilde{K} R}F_\alpha(\varphi(t))\geq 0\,.$$
This of course, gives us
\begin{equation}\label{EL1.1I}
(\varphi^\prime(t))^4\leq e^{\frac{\tilde{K}R}{2}} \alpha^4 + 
e^{\frac{\tilde{K}R}{2}}\frac{4}{3} F_\alpha(\varphi(t)), \quad t\in (R_1, R)\,.
\end{equation}
Now, without any loss of generality, we may take $\varepsilon\in (0, 1)$. 
Therefore, at the beginning, if we choose $M_2$ large enough to satisfy
$$ \left[e^{\frac{\tilde{K}R}{2}} \alpha^4 + 
\frac{4}{3} e^{\frac{\tilde{K}R}{2}} \max_{s\in [0, 1]}F_\alpha(s)\right]^{\nicefrac{1}{4}}< M_2, $$
 possibility (b) in 
\eqref{EL1.1H} can not occur before (a). In other words, if we have $R_1>R/2$, then 
(a) is the only possibility. Thus it is enough to consider (a).
Restrict $\varepsilon<\delta$ where $\delta$ is given by
\eqref{EL1.1A}. We can further restrict $\varepsilon$ to satisfy
$F_\alpha(\varepsilon)<\delta$. 
Choose $\alpha$ small enough so that 
$\widehat{\varepsilon}= F_\alpha^{-1}(\alpha^4)<\varepsilon$. Then we can find
$\hat{R}\in (R_1, R)$ satisfying $\varphi(\hat{R})=\hat{\varepsilon}$ and
$\varphi(t)\geq \hat\varepsilon$ in $(R_1, \hat{R})$. Then
$$F_\alpha(\varphi(t))\geq F_\alpha (\hat\varepsilon)=\alpha^4\quad \text{in}
\; (R_1, \hat{R}).$$
Using \eqref{EL1.1I} we then obtain
$$ -\varphi^\prime(t) \leq \left(\frac{7}{3}\right)^\frac{1}{4}\,
e^{\frac{\tilde{K} R}{8}} [F_\alpha(\varphi(t))]^\frac{1}{4}
\quad \text{in}\; (R_1, \hat{R}).$$
Integrating both sides we have
$$
\int^{\hat{R}}_{R_1} \frac{-\varphi^\prime(t)}{[F_\alpha(\varphi(t))]^\frac{1}{4}}\, \D{t}
\leq \, \left(\frac{7}{3}\right)^\frac{1}{4}\,
e^{\frac{\tilde{K} R}{8}} \frac{R}{2},$$
which in turn, gives
\begin{equation}\label{AB3}
\int^{\varepsilon}_{\hat\varepsilon} \frac{1}{[F(t)-\alpha t]^\frac{1}{4}}\, \D{t}=
\int^{\varepsilon}_{\hat\varepsilon} \frac{1}{[F_\alpha(t)]^\frac{1}{4}}\, \D{t}
\leq \, \left(\frac{7}{3}\right)^\frac{1}{4}\,
e^{\frac{\tilde{K} R}{8}} \frac{R}{2}.
\end{equation}
Since $\alpha\to 0$ implies $\hat{\varepsilon}\to 0$, using monotone convergence 
theorem we note that
$$\int^{\varepsilon}_{\hat\varepsilon} \frac{1}{[F(t)-\alpha t]^\frac{1}{4}}\, \D{t}
\to \int^{\varepsilon}_{0} \frac{1}{[F(t)]^\frac{1}{4}}\, \D{t}, \quad \text{as}\; \alpha\to 0.$$
But the limit is $\infty$ by \eqref{EL1.1A}. This is a contradiction to \eqref{AB3}.
Hence (a) in \eqref{EL1.1H} is also not possible for small enough $\alpha$.
Thus $R_1\leq R/2$ which completes the proof.
\end{proof}

Now we are ready to prove our main results. We start with the proof of Theorem~\ref{T1.1}.
\begin{proof}[{\bf Proof of Theorem~\ref{T1.1}}]
We only provide a proof for (a) and the proof for (b) would be analogous.
Suppose, on the contrary, that the set $\{x\in\cO\; :\; v(x)=0\}$ is non-empty. Then
since $v\neq 0$, we can find a ball $B(x_0, R)\subset \cO$ such that
$v>0$ in $B(x_0, R)$ and 
$\overline{B(x_0, R)}\cap\{x\in\cO\; :\; v(x)=0\}\neq\emptyset$. Without loss of generality, assume that $R\in (0, 1)$. Choose $\varepsilon< \min_{\overline{B}(x_0, R/2)} v$.
Using Lemma~\ref{L1.1} we now find a twice continuously differentiable function
$\varphi$ satisfying
\begin{align}
(\varphi^\prime)^2\varphi^{\prime\prime} + K (\varphi^\prime)^3 - f(\varphi)+\alpha&=0\quad
\text{in}\; ({R}/{2}, R+\varepsilon_1),\quad
\varphi^\prime(R)=\alpha, \quad \varphi(R)=0\,,\label{ET2.1C}
\\
&0<\varphi<\varepsilon, \quad \varphi^\prime<0\quad \text{in} \; ({R}/{2}, R)\,,\label{ET2.1D}
\end{align}
for some $\varepsilon_1>0$ and $\alpha<0$. Let $u(x)= \varphi(|x-x_0|)$. Then in $B^c(x_0, R/2)$
we have
\begin{align*}
Du(x)&= \frac{x-x_0}{|x-x_0|}\varphi^\prime(|x-x_0|)
\\
\partial_{x_ix_j}u &= \frac{(x-x_0)_i(x-x_0)_j}{|x-x_0|^2}\varphi^{\prime\prime}(|x-x_0|)+ \varphi^\prime(|x-x_0|)\left(\frac{\delta_{ij}}{|x-x_0|}-\frac{(x-x_0)_i(x-x_0)_j}{|x-x_0|^3}\right).
\end{align*}
Using \eqref{ET2.1C}-\eqref{ET2.1D} we then have
\begin{align}\label{ET2.1E}
\sL_1 u - K |Du|^3 - f(u)
& =
(\varphi^\prime)^2(|x-x_0|)\varphi^{\prime\prime}(|x-x_0|)
- K |\varphi^{\prime}|^3 - f(\varphi)\nonumber
\\
&= (\varphi^\prime)^2(|x-x_0|)\varphi^{\prime\prime}(|x-x_0|)
+ K (\varphi^{\prime})^3 - f(\varphi)=-\alpha>0.
\end{align}
Using Lemma~\ref{L2.1} we then have $u\leq v$ for $R/2\leq|x-x_0|\leq R$. Also,
note that for $R\leq |x-x_0|\leq R+\varepsilon_1$, $u(x)\leq 0$. Thus, $u$ touches
$v$ from below at some point, say $z$, on the sphere $|x-x_0|=R$. Applying the definition of viscosity solution we must have
$$\sL_1 u(z) - K |Du(z)|^3 - f(u(z))=\sL_1 u(z) - K |Du(z)| - f(v(z))\leq 0,$$
which contradicts \eqref{ET2.1E}. Thus $\{x\in\cO\; :\; v(x)=0\}=\emptyset$, completing
the proof.
\end{proof}

Next we prove Theorem~\ref{T1.2}.

\begin{proof}[{\bf Proof of Theorem~\ref{T1.2}}]
First we consider (a). We start by assuming that $\lim_{|x|\to\infty} u(x)=0$ and
establish a compact support principle.
For this proof we borrow the ideas from \cite{FMQ}. The main idea is to construct
a nonnegative super-solution to \eqref{EqnA} with a compact support. To do so we again use the ODE technique.
Note that by monotonicity of $f$ we have $F(a t)\leq a F(t)$ for every $a\in [0, 1]$, 
since
$$F(a t)=\int_0^{at} f(s) \D{s}=a\int_0^t f(as) \D{s}\leq a\int_0^t f(s) \D{s}=aF(t).$$
Thus, by \eqref{ET2.2A0}, we have
$$\int_0^1 \frac{1}{\Gamma^{-1}(4^{-1}F(s))}\, \D{s}<\infty.$$
Define a continuous function $\varphi$ by
$$t=\int_0^{\varphi(t)} \frac{1}{\Gamma^{-1}(4^{-1}F(s))}\, \D{s}\,. $$
Note that $\varphi$ is strictly increasing with $\varphi(0)=0$. Also,
\begin{equation}\label{ET2.2C0}
1=\frac{\varphi'(t)}{\Gamma^{-1}(4^{-1}F(\varphi(t)))}
\quad \Rightarrow\quad \Gamma(\varphi'(t))=\frac{1}{4} F(\varphi(t)).
\end{equation}
Since $\Gamma, F, \varphi$ are strictly increasing, we have $\varphi'$ strictly increasing and $\varphi'(0)=0$. Hence for $t\leq 1$ we have
$\varphi(t)=\int_0^t \varphi'(s)\, \D{s}\leq \varphi'(t)$.
It is also evident from \eqref{ET2.2C0} that 
$\varphi'$ is continuously differentiable for $t>0$. Therefore, using \eqref{ET2.2C0}
and the fact $G$ is nondecreasing, we obtain for $t\in [0, 1]$ that
\begin{align*}
G(\varphi'(t))\varphi'(t)\leq \int_{\varphi'(t)}^{2\varphi'(t)} G(\varphi'(s))\, \D{s}
&\leq \Gamma(\varphi'(t))=\frac{1}{4} F(\varphi(t))\leq \frac{1}{4} f(\varphi(t))\varphi(t)\leq \frac{1}{4} f(\varphi(t))\varphi'(t),
\end{align*}
which in turn, implies
\begin{equation}\label{ET2.2D0}
G(\varphi')\leq \frac{1}{4} f(\varphi(t))\quad \text{for all $t>0$ small}.
\end{equation}
Since $\varphi^{\prime\prime}\geq 0$ for $t>0$, differentiating \eqref{ET2.2C0} we have
$$(\varphi'(t))^3\varphi^{\prime\prime}(t) 
\leq (\Gamma(\varphi(t))'=\frac{1}{4} f(\varphi(t))\varphi'(t),$$
giving us
\begin{equation}\label{ET2.2E0}
(\varphi'(t))^2\varphi^{\prime\prime}(t)\leq \frac{1}{4} f(\varphi(t))
\quad \text{for all $t>0$ small}.
\end{equation}
Combining \eqref{ET2.2D0}-\eqref{ET2.2E0} we find $r_\circ>0$ such that 
\begin{equation}\label{ET2.2F0}
(\varphi'(t))^2\varphi^{\prime\prime}(t) + G(\varphi'(t))-2^{-1}f(\varphi(t))
\leq 0 \quad \text{for all}\; t\in (0, r_\circ),
\end{equation}
and $\varphi(0)=\varphi'(0)=0$. We extend $\varphi$ on $(-\infty, 0]$ by setting
$\varphi(t)=0$ for $t\leq 0$. It is easily seen that $\varphi$ is continuously
differentiable in $(-\infty, r_\circ)$. Now for any $R>0$, we let 
$v(x)=\varphi(R+r_\circ-|x|)$ for $|x|\geq R$. Using \eqref{ET2.2F0} and the calculations in \eqref{ET2.1E}
we see that for $ R<|x|< R+r_\circ$ we have
\begin{align}\label{ET2.2G0}
\sL_1 v + G(|Dv|) - 2^{-1}f(v)
 =
(\varphi^\prime)^2(R+r_\circ-|x|)\varphi^{\prime\prime}(R+r_\circ-|x|)
+ G(\varphi^{\prime}) - 2^{-1}f(\varphi)
\leq 0.
\end{align}
We claim that 
\begin{equation}\label{ET2.2H0}
\sL_1 v + G(|Dv|) - 2^{-1}f(v)\leq 0 \quad \text{for}\; |x|> R\,,
\end{equation}
in viscosity sense. When $ R<|x|< R+r_\circ$, \eqref{ET2.2H0} follows
from \eqref{ET2.2G0}. Again, for $|x|> R+r_\circ$, \eqref{ET2.2H0} is evident  as
$v$ is identically zero there.
So we consider the case where $|x|= R+r_\circ$. Let $\chi$ be a $\cC^2$ test with
$\chi\prec_{x} v$. Since $v$ is $\cC^1$ we have $D\chi(x)=Dv(x)=0$. Hence
$$\sL_1\chi(x) + G (|D\chi(x)|)- 2^{-1}f(v(x))= 0,$$
implying $v$ is supersolution.  This 
gives us \eqref{ET2.2H0}.

Now we complete the proof. Let $\beta=\varphi(r_\circ)>0$. Since $u(|x|)\to 0$
as $|x|\to\infty$, we find $R$ so that $u(x)<\beta $ for $|x|\geq R$.
Define 
$v_\epsilon(x)=v(x) + \epsilon$. Since $f$ is non-decreasing, we get from \eqref{ET2.2H0} that
$$\sL_1 v_\epsilon + G (|Dv_\epsilon|)-f(v_\epsilon) 
\leq 2^{-1}f(v)-f(v_\epsilon)\leq -2^{-1}f(v_\epsilon)\,,
\quad |x|> R.$$
Now we choose $R_\epsilon>R+r_\circ$ large enough so that $u(x)<\epsilon$
for $|x|\geq R_\epsilon$. Applying Lemma~\ref{L2.1} in $\{R<|x|<
R_\epsilon\}$ we obtain $u\leq v_\epsilon=v+\epsilon$ for $|x|\geq R$.
Now let $\epsilon\to 0$ to conclude that $u\leq v$ for $|x|\geq R$
which implies $u(x)=0$ for $|x|\geq R+r_\circ$. This completes the proof of
(a) under the assumption $\lim_{|x|\to\infty} u(x)=0$.

Now consider a bounded solution $u$ to
\eqref{EqnA} and we show that $\lim_{|x|\to\infty} u(x)=0$. Then the conclusion 
of (a) follows from the first part of the proof.
Suppose, on the contrary, that $\limsup_{|x|\to\infty} u(x)=M>0$ and $f(M)=3\kappa$.
Given $x_0\in\RN$, we define $\xi(x)=r^{-2}|x-x_0|^2-1$. Then 
$\xi<0$ in $B(x_0, r)$ and vanishes at the boundary $\partial B(x_0, r)$. Also,
$$\sL_1\xi + G(|D\xi|)= 8 r^{-6}|x-x_0|^2 + G(2r^{-2}|x-x_0|)<\kappa
\quad \text{in}\; B(x_0, r),$$
provided $r$ is large. Now choose a point $x_0\in\cO$ such that $B(x_0, r)\subset\cO$
and $u(x_0)> M-\epsilon$ where $\epsilon\in (0, 1/2)$ is small enough to satisfy 
$f(M-\epsilon)>2\kappa$.  We may also assume that $\sup_{\overline{B(x_0, r)}} u< M+ 1/2$.
 Now we translate $\xi$ so that it touches $u$ in $B(x_0, r)$ from above.
To do so, we define
$$\beta=\inf\{\gamma\in[M-2\epsilon, M+2]\; :\; \gamma+\xi > u\; \; \text{in}\; B(x_0, r)\}.$$
Clearly, $\beta\geq M+1-\epsilon$ since $M+1-\epsilon+\xi(x_0)=M-\epsilon< u(x_0)$.
Hence $v(x):=\beta+\xi(x)=\beta\geq M+1-\epsilon> u(x)$ on $\partial B(x_0, r)$.
$u$ being upper-semicontinuous, $v$ must touch $u$ inside $B(x_0, r)$, say at a point $z\in B(x_0, r)$.
Then applying the definition of viscosity solution we obtain
$$f(v(z))=f(u(z))\leq \sL_1 v + G(|Dv|)\leq \kappa.$$
Since $v(z)\geq M+1-\epsilon+\xi(z)\geq M-\epsilon$, we get from above that $\kappa\geq f(v(z))\geq f(M-\epsilon)>2\kappa$ which is a contradiction. Thus we must have $\lim_{|x|\to\infty} u(x)=0$. Hence the proof.

Next we consider (b). The proof is similar to (a). Following a similar argument 
of \eqref{ET2.2F0} we would obtain
\begin{equation*}
\varphi^{\prime\prime}(t) + G(\varphi'(t))-2^{-1}f(\varphi(t))
\leq 0 \quad \text{for}\; t\in (0, r_\circ)\,.
\end{equation*}
It can be easily seen from this equation that 
$\lim_{t\to 0+}\varphi^{\prime\prime}(t)=\varphi^{\prime\prime}(0)=0$. Thus the extension of $\varphi$ is twice continuously differentiable in $(-\infty, r_\circ)$.
Now we can follow the arguments of (a) to complete the proof.
\end{proof}

Finally, we prove Theorem~\ref{T1.3}.
\begin{proof}[{\bf Proof of Theorem~\ref{T1.3}}]
As before, we only prove (a) and the proof for (b) would be analogous.
Let 
$$R= \int_0^1 \frac{1}{[H(s)]^{\nicefrac{1}{4}}}\, \D{s},$$
where $H(t)=\int_0^t h(s) \D{s}$ and $h(s)=4\kappa f(s)$ for some $\kappa>0$ to 
be chosen later.
We define a continuous function $\varphi:[0, R]\to [0, \infty)$ by
$$r=\int_{\varphi(r)}^1 \frac{1}{[H(s)]^{\nicefrac{1}{4}}}\, \D{s}.$$
 It is evident that $\varphi$ takes values in $[0, 1]$ and is strictly decreasing.
Differentiating  we obtain
$$\frac{-\varphi'(r)}{[H(\varphi(r))]^{\nicefrac{1}{4}}}=1 \quad \text{for}\; 0<r< R.$$
Since $\varphi'\neq 0$ in $(0, R)$, differentiating once again we get
\begin{equation}\label{ET2.2A}
(\varphi')^2 \varphi^{\prime\prime} - \frac{1}{4}h(\varphi(r))=0\quad \text{in}\; (0, R).
\end{equation}
Since $\varphi(R)=\varphi^\prime(R)=0$, from \eqref{ET2.2A} we have
\begin{equation}\label{ET2.2B}
\varphi(r)=\int_r^R \left[\int_t^R \frac{3}{4} h(\varphi(s)) \D{s}\right]^{\nicefrac{1}{3}}\, \D{t}\quad r\in (0, R]\,.
\end{equation}
Choose $\delta>0$ small enough so that $8e^{-3K\delta}\geq 1$, where $K$ is same as in Theorem~\ref{T1.3}. Now consider a map $T:\cC[R-\delta, R]\to \cC[R-\delta, R]$
given by
\begin{equation}\label{ET2.2C}
(Tg)(t) = \int_t^{R} \left[\int_{s}^{R} 6 e^{3K(s-\zeta)} h(g(\zeta))\, \D\zeta\right]^{\nicefrac{1}{3}} \D{s}\,.
\end{equation}
It is easily seen that $T$ is a continuous function. Also, if $g\geq \varphi$, then
since $h$ is non-decreasing using \eqref{ET2.2B} we get
$$(Tg)(t)\geq \int_t^{R} \left[\int_{s}^{R} 6 e^{-3K\delta} h(\varphi(\zeta))\, \D\zeta\right]^{\nicefrac{1}{3}} \D{s}
\geq \int_t^{R} \left[\int_{s}^{R} 3 h(\varphi(\zeta))\, \D\zeta\right]^{\nicefrac{1}{3}} \D{s}=\varphi(t)\,.$$
Denote by $M=\sup_{s\in [0, 1]} h(s)$. Then, restricting $\delta$ small enough
we see that if $\sup_{t\in [R-\delta, R]}|g(t)|\leq 1$ then
\begin{align*}
|(Tg)(t)|\leq (6M)^{\frac{1}{3}}\int_t^{R} (R-s )^{\nicefrac{1}{3}} \D{s}
= \frac{3}{4}(6M)^{\frac{1}{3}} (R-t)^{\frac{4}{3}}
=\frac{3}{4}(6M)^{\frac{1}{3}} \delta^{\frac{4}{3}}\leq 1.
\end{align*}
Furthermore, 
$$|(Tg)(t_1)-(Tg)(t_2)|\leq (6M\delta)^{\frac{1}{3}}|t_1-t_2|.$$
Thus, letting
$$\cA=\{g\in \cC[R-\delta, R]\; :\; g\geq \varphi,\; \max_{[R-\delta, R]}|g|\leq 1,
\; |g(t_1)-g(t_2)|\leq (6M\delta)^{\frac{1}{3}}|t_1-t_2|\quad\forall\, t_1, t_2\in [R-\delta, R]\},$$
we note that $T:\cA\to\cA$. Therefore, by Schauder fixed point theorem, 
$T$ has a fixed point $\psi$ in $\cA$.
In particular, we get from \eqref{ET2.2C}
$$\psi(t) = \int_t^{R} \left[\int_{s}^{R} 6 e^{3K(s-\zeta)} h(\psi(\zeta))\, \D\zeta\right]^{\nicefrac{1}{3}} \D{s}\,.
$$
This of course, implies $\psi(R)=0$. Differentiating we obtain
$$-(\psi^\prime(t))^3 = \int_{t}^{R} 6 e^{3K(t-\zeta)} h(\psi(\zeta))\, \D\zeta
\quad t\in (R-\delta, R).$$
Thus $D_{-}\psi(R)=0$ and differentiating the above equation we obtain
\begin{equation}\label{ET2.2D}
(\psi^\prime(t))^2\psi^{\prime\prime}(t)-K(\psi^\prime(t))^3 - 2h(\psi(t))=0
\quad \text{for}\; t\in (R-\delta, R).
\end{equation}
Extend $\psi$ in $(R, \infty)$ be setting $\psi(t)=0$ for $t\geq R$. Note that
$\psi$ is continuously differentiable in $(R-\delta, \infty)$ and $\psi'<0$ in
$(R-\delta, R)$.

Now we let $\kappa=\frac{1}{8}$. Let $r_\circ=R-\delta-1$ and define
 $v(x)=\psi(|x|+r_\circ)$. Using \eqref{ET2.2D} and the calculations in \eqref{ET2.1E}
we see that for $ 1<|x|< 1+\delta$ we have
\begin{align}\label{ET2.2E}
\sL_1 v + K |Dv|^3 - 2f(v)
& =
(\psi^\prime)^2(|x|+r_\circ)\psi^{\prime\prime}(|x|+r_\circ)
+ K |\psi^{\prime}|^3 - f(\psi)\nonumber
\\
&= (\psi^\prime)^2(|x|+r_\circ)\psi^{\prime\prime}(|x|+r_\circ)
- K (\psi^{\prime})^3 - 2h(\psi)=0.
\end{align}
We claim that 
\begin{equation}\label{ET2.2F}
\sL_1 v + K |Dv|^3 - f(v)=0 \quad \text{for}\; |x|> 1\,,
\end{equation}
in viscosity sense. When $ 1<|x|< 1+ \delta$, \eqref{ET2.2F} follows
from \eqref{ET2.2E}. Again, for $|x|> 1+ \delta$, \eqref{ET2.2F} is evident.
So we consider the case where $|x|= 1+ \delta$. Let $\chi$ be a $\cC^2$ test with
$v\prec_{x}\chi$. Since $v$ is $\cC^1$ we have $D\chi(x)=Dv(x)=0$. Hence
$$\sL_1\chi(x) + K |D\chi(x)|^3-f(v(x))=0,$$
implying $v$ is subsolution. Similarly, we show that $v$ is a supersolution. This 
gives us \eqref{ET2.2F} and this completes the proof.
\end{proof}

\subsection*{Acknowledgement}

The author is grateful to the referee for his/her careful reading and comments.
The research of Anup Biswas was supported in part by DST-SERB grants EMR/2016/004810 and MTR/2018/000028 and a SwarnaJayanti fellowship.

\end{document}